\documentclass[11pt, letterpaper]{amsart}
\usepackage[plainpages=false,pdfpagelabels,
colorlinks=true,linkcolor=blue,
citecolor=blue,filecolor=blue,      
urlcolor=blue,hypertexnames=false]{hyperref}
\usepackage{amssymb}
\usepackage{graphics}
\usepackage{graphicx}               
\usepackage{latexsym}
\usepackage{amsmath}
\usepackage{amssymb,amsthm,amsfonts}
\usepackage{amscd}
\usepackage[arrow, matrix, curve]{xy}
\usepackage{syntonly}
\usepackage{tikz}
\usepackage[small,nohug,heads=vee]{diagrams}
\diagramstyle[labelstyle=\scriptstyle]

\theoremstyle{plain}
\newtheorem{thm}{Theorem}[section]
\newtheorem{theorem}[thm]{Theorem}
\newtheorem*{theorem*}{Theorem}
\newtheorem{lemma}[thm]{Lemma}

\newtheorem{proposition}[thm]{Proposition}

\theoremstyle{definition}

\newtheorem{remark}[thm]{Remark}

\newtheorem{definition}[thm]{Definition}

\newtheorem{conjecture}[thm]{Conjecture}
\numberwithin{equation}{thm}

\title[Chevalley restriction theorem in characteristic $p$]{A higher-dimensional Chevalley restriction theorem for classical groups in characteristic $p$}

\author[]{Xiaopeng Xia}
\address{School of Mathematical Sciences, University of Science and Technology of China\\ No. 96 Jinzhai Road, Hefei, Anhui 230026, P. R. China}
\email{xpxia@mail.ustc.edu.cn}

\date{}

\begin{document}
 \begin{abstract}
 We establish a theorem concerning the commuting scheme in characteristic 
$p$. As a significant application of this theorem, we derive an explicit lower bound for the characteristic 
$p$, ensuring the validity of the higher-dimensional Chevalley restriction theorem for classical groups.

 \end{abstract}
 \maketitle  

\section{Introduction}

Let $G$ be a reductive group over an algebraically closed field $\mathbb{K}$ with Lie algebra $\mathfrak{g}$. For an integer $d\geq 2$,  let  $\mathfrak{C}^d_{\mathfrak{g}}\subset \mathfrak{g}^d$ be the  commuting scheme, which is defined as the scheme-theoretic fiber of the commutator map over the zero 
\[
\mathfrak{g}^d\rightarrow \prod\limits_{i<j} \mathfrak{g}, \ \ (x_1,\cdots, x_d)\mapsto \prod_{i<j}[x_i, x_j].
\] 
Its underlying variety (the reduced induced closed subscheme) $\mathfrak{C}^d_{\mathfrak{g},red}$ is called the commuting variety. As a set, $\mathfrak{C}^d_{\mathfrak{g},red}$ consists of $d$-tuples $(x_1,\cdots, x_d)\in \mathfrak{g}^d$ such that $[x_i,x_j]=0$, for all  $1\leq i,j\leq d$. It is a long-standing open question whether or not $\mathfrak{C}^d_{\mathfrak{g}}$ is reduced, that is $\mathfrak{C}^d_{\mathfrak{g}}=\mathfrak{C}^d_{\mathfrak{g},red}$.  When the characteristic of $\mathbb{K}$ ($\rm{char} \ \mathbb{K}$ for short) is zero, Charbonnel [\ref{Charbonnel}] recently claims a proof for $\mathfrak{C}^2_{\mathfrak{g}}=\mathfrak{C}^2_{\mathfrak{g},red}$. Although there is no adequate evidence to expect $\mathfrak{C}^d_{\mathfrak{g}}$ is reduced for general $d$, one can study the categorical quotient $\mathfrak{C}^d_{\mathfrak{g}}/\!\!/G$ and ask the same question. Here $G$ acts on $\mathfrak{g}^d$ by the diagonal adjoint action, and the action leaves $\mathfrak{C}^d_{\mathfrak{g}}$ stable. 

Let $T$ be a maximal torus of $G$ and $\mathfrak{t}$ be the Lie algebra of $T$. Then the Weyl group $W:=N_G(T)/T$ acts on $\mathfrak{t}^d$ diagonally. The embedding $\mathfrak{t}^d\hookrightarrow \mathfrak{g}^d$ factors through $\mathfrak{C}^d_{\mathfrak{g}}$ and induces the natural morphism \[\Phi: \mathfrak{t}^d/\!\!/W \rightarrow \mathfrak{C}^d_{\mathfrak{g}}/\!\!/ G.\]In studying the Hitchin morphism from the moduli stack of principle $G$-Higgs bundles on a proper smooth variety $X$ of dimension $d\ge 2$, Chen and Ng\^{o} \cite{Chen-Ngo} are led to

\begin{conjecture}[Chen-Ng\^{o}]\label{conj}
The morphism $\Phi: \mathfrak{t}^d/\!\!/W\rightarrow \mathfrak{C}^d_{\mathfrak{g}}/\!\!/G$ is an isomorphism. 
\end{conjecture}

When $d=1$ and $\rm{char} \ \mathbb{K} =0$, Conjecture \ref{conj} is simply the classical Chevalley restriction theorem. Since in the context of Higgs bundles, $d$ is the dimension of the underlying variety $X$, we consider the conjecture as a higher-dimensional analogue of Chevalley restriction theorem. Note when $d=2$ and $\rm{char} \ \mathbb{K} =0$, this conjecture is a special  case (degree zero part) of a more general conjecture proposed by Berest et al. [\ref{Berest et al.}]. 

If $\rm{char} \ \mathbb{K} =0$, Conjecture \ref{conj} is  known  to hold for  $G=GL_n(\mathbb{K})$ (Vaccarino [\ref{Vaccarino}], Domokos [\ref{Domokos}], and later Chen-Ng\^{o} [\ref{Chen-Ngo}] independently; see also Gan-Ginzburg [\ref{Gan-Ginzburg}] for case $d=2$), for $G=Sp_{n}(\mathbb{K})$ (Chen-Ng\^{o} [\ref{Chen-Ngo-Symplectic}]), for $G=O_n(\mathbb{K}),SO_n(\mathbb{K})$ (Song-Xia-Xu [\ref{S-X-X}]) and for connected reductive groups (Li-Nadler-Yun [\ref{L-N-Y:commuting stack}] for case $d=2$). A weaker version $\mathfrak{t}^d/\!\!/W\xrightarrow{\sim} \mathfrak{C}^d_{\mathfrak{g},red}/\!\!/G$ is proved by Hunziker [\ref{Hunziker}] if $G$ is of type $A,B,C,D$ or $G_2$.

If $\rm{char}\ \mathbb{K} >0$, Conjecture \ref{conj} is largely open. However, the weaker version $\mathfrak{t}^d/\!\!/W\xrightarrow{\sim} \mathfrak{C}^d_{\mathfrak{g},red}/\!\!/G$ is proved by Vaccarino [\ref{Vaccarino}] for $G=GL_n(\mathbb{K})$ and by Song-Xia-Xu [\ref{S-X-X}] for $G=Sp_n(\mathbb{K}),O_n(\mathbb{K}),SO_n(\mathbb{K})$.

The main purpose of the article is to prove Conjecture \ref{conj} for classical groups in case that the characteristic $p$ is greater than an explicit bound. To be more precise, our main result is the following (see
Theorems \ref{main thm: good fil} and \ref{main thm})

\begin{theorem}\label{main thm in introduction}
Suppose $n\geq 2$, $d\geq 1$, and  $G$ is one of groups $GL_n(\mathbb{K})$, $O_n(\mathbb{K})$, $SO_n(\mathbb{K})$ or $Sp_n(\mathbb{K})$. 
\begin{enumerate}
\item Let $R=\mathbb{K}[M_n(\mathbb{K})^d]$, $\Bar{R}=\mathbb{K}[\mathfrak{g}^d]$, $I=\ker(R\to \mathbb{K}[\mathfrak{C}^d_{\mathfrak{g}}])$ and $\Bar{I}=\ker(\Bar{R}\to\mathbb{K}[\mathfrak{C}^d_{\mathfrak{g}}])$.
If $\mathrm{char}\ \mathbb{K}>2(\mathrm{reg}_{\Bar{R}}(\Bar{I})-1)(n-1)+\frac{d(n^3-n)}{3}$ and $G$ is connected, then $G$-modules $\Bar{I}$ and $I$ have good filtrations.
\item If $\mathrm{char}\ \mathbb{K}>(12)^{2^{d\cdot\dim(\mathfrak{g})-4}}(2n-2)+\frac{d(n^3-n)}{3}$, then $\Phi: \mathfrak{t}^d/\!\!/W\rightarrow \mathfrak{C}^d_{\mathfrak{g}}/\!\!/G$ is an isomorphism. 
\end{enumerate} 
\end{theorem}

\subsection*{Acknowledgments}
 X. X. would like to thank Jinxing Xu for helpful discussions related to this work. During the preparation of the article, X. X. was partially supported by the Innovation Program for Quantum Science and Technology (2021ZD0302902).

\section{Notations and preliminaries}

In this section we fix some notations and then record some useful definitions and propositions that will be used frequently in the subsequent sections.  

Throughout this paper, $\mathbb{K}$ is an algebraically closed field with $\rm{char} \ \mathbb{K} >2$. We denote $M_n(R)$ as the set of $n\times n$ matrices over a commutative ring $R$, and for a matrix $M$, we denote $M^t$ as the transpose, and $M(i, j)$ as the $(i, j)$-entry of $M$ . We denote the coordinate ring of an affine $\mathbb{K}$-scheme $X$ by $\mathbb{K}[X]$. In particular, 
if $V$ is a $\mathbb{K}$-linear space, $\mathbb{K}[V]$ means the $\mathbb{K}$-algebra of polynomial functions on $V$. If a group $G$ acts $\mathbb{K}$-linearly on $V$, then $G$ acts naturally  on $\mathbb{K}[V]$ by $g\cdot f \ (v) =f(g^{-1}\cdot v)$, for $g\in G, \ f\in \mathbb{K}[V], \ v\in V$.  The $\mathbb{K}$-algebra  of $G$-invariant polynomials on $V$ is denoted by  $\mathbb{K}[V]^G$. 

We will need the following lemma. 

\begin{lemma}\label{lemma:max}
    Given a positive integer $n>1$. For any integer $0<m\leq n$, let 
    \[
    F_m(x_1,\dots,x_{2m})=\sum_{i=1}^m (x_i-x_{i-1})x_{m+i}(n-x_{m+i})+(x_{m+i}-x_{m+i+1})x_{i}(n- x_{i})
    \]
    be a function of $2m$ variables on $\mathbb{R}^m$, where $x_0=x_{2m+1}=0$, and let $V_m=\{(a_1,\dots,a_{2m})\in \mathbb{Z}^{2m}\mid 0<a_1<\dots<  a_{m}\leq n \geq  a_{m+1}>\dots > a_{2m}> 0\}$. Then 
    \[
    \max\{F_m(\alpha)\mid  1\leq m\leq n,\alpha\in V_m\}= \frac{n^3-n}{3}.
    \]  
\end{lemma}
\begin{proof}
    Let $M=\max\{F_m(\alpha) \mid 1\leq m\leq n,\alpha\in V_m\}$. For any $1\leq m\leq n$, consider these equations
    \[
    \begin{cases}
        \frac{\partial F_m}{\partial x_{m+i}}=(x_i-x_{i-1})(2n-2x_{m+i}-x_i-x_{i-1})=0, & \forall 1 \leq i\leq m \\
        0=x_0< x_1 < \dots< x_{m}\leq n
    \end{cases}
    \]
    By solving the above system of equations, we see that
    \[
    x_{m+i}=\frac{2n-x_i-x_{i-1}}{2},1\leq i\leq m.
    \]
    Since $\frac{\partial^2 F_m}{\partial x_{m+i}\partial x_{m+j}}=-2\delta_{ij}$ for any $1\leq i,j\leq m$, where $\delta_{ij}$ is the Kronecker symbol, we have
    \begin{align}\label{equ:F-F(a)=-XAX^t}
    \begin{split}
          F_m(x_1,\dots,x_{2m}) &- F_m(x_1,\dots,x_{m},\frac{2n-x_1-x_{0}}{2},\dots,\frac{2n-x_m-x_{m-1}}{2}) \\
      & = -\sum_{i=1}^m (x_{m+i}-\frac{2n-x_i-x_{i-1}}{2})^2\leq 0.
    \end{split}
    \end{align}
    Let $H_m(x_1,\dots,x_{m})=F_m(x_1,\dots,x_{m},f_1(x_1,\dots,x_{m}),\dots,f_m(x_1,\dots,x_{m}))$ be a function of $m$ variables on $\mathbb{Z}^m$ and $W_m=\{(a_1,\dots,a_{m})\in \mathbb{Z}^{m}\mid 0<a_1<\dots<  a_{m}\leq n\}$, where $f_i(x_1,\dots,x_{m})=\frac{2n-x_i-x_{i-1}}{2}+\frac{1-(-1)^{x_i+x_{i-1}}}{4}$ for any $1\leq i\leq m$ and $x_0=0$.
    Note that if $m<n$ and  $(a_1,\dots,a_m)\in \mathbb{Z}^m$, then
    \begin{equation}\label{equ:H_m(a)=H_m+1(0,a)}
        H_m(a_1,\dots,a_m)=H_{m+1}(0,a_1,\dots,a_m)
    \end{equation}
    If $\alpha\in W_m$, then $n \geq  f_1(\alpha)>\dots > f_{m}(\alpha)> 0$ and $f_i(\alpha)\in \mathbb{Z}$ for any $1\leq i\leq m$. So it follows from (\ref{equ:F-F(a)=-XAX^t}) that 
    \[
    M=\max\{H_m(\alpha)\mid  1\leq m\leq n,\alpha\in W_m\}.
    \]

    Given a positive integer $1\leq m\leq n$, we will compare $H_m(\alpha)$ and $H_m(\alpha')$ in two cases for any $\alpha=(a_1,\dots,a_m),\alpha'=(a'_1,\dots,a'_m)\in\mathbb{Z}^{m}$.
    If $a_1<a'_1$, $a_2>1$ and $a_i=a'_i$ for any $2\leq i\leq m$, then 
    \begin{align*}
         & H_m(\alpha)-H_m(\alpha') \\
              = &\begin{cases}
              \frac{a_2(a'_{1}-a_1)}{4}\left(\frac{(-1)^{a_1}-(-1)^{a'_1}}{2(a'_{1}-a_1)}+a_1+a'_{1}-a_2 \right),  \ \textmd{ if } 2\mid  a_2;\\
             \frac{a_2(a'_{1}-a_1)}{4}\left(\frac{(-1)^{a'_1}-(-1)^{a_1}}{2(a'_{1}-a_1)}-\frac{a'_1(-1)^{a'_1}-a_1(-1)^{a_1}}{a_2(a'_{1}-a_1)}+a_1+a'_{1}-a_2 \right),  \ \textmd{ if } 2\nmid a_2.
             \end{cases}
    \end{align*}
    Hence
    \begin{equation}\label{equ:a_1}
        \begin{cases}
             H_m(\alpha)< H_m(\alpha'),  \ \textmd{ if }  a_1<a'_1< \frac{a_2}{2} \textmd{, or } 4\mid  a_2 \textmd{ and } a_1+1=a'_1=\frac{a_2}{2} ;\\
             H_m(\alpha)> H_m(\alpha'),  \ \textmd{ if }  \frac{a_2}{2}< a_1<a'_1 \textmd{, or } 4\mid  a_2 \textmd{ and } \frac{a_2}{2}=a_1=a'_1-1;\\
             H_m(\alpha)= H_m(\alpha'),  \ \textmd{ if }  a_2\equiv 2 \pmod{4} \textmd{ and } a_1,a'_1\in \{\frac{a_2}{2},\frac{a_2}{2}\pm 1\};\\
             H_m(\alpha)=H_m(\alpha') ,  \ \textmd{ if }  2\nmid a_2 \textmd{ and } a_1+1=a'_1=\frac{a_2+1}{2}.
        \end{cases}
    \end{equation}
    If there is $1< r <m$ such that $a_r=a_{r-1}+1=a'_{r}-1\leq a_{r+1}-2$ and $a_i=a'_i$ for any $1\leq i\neq r\leq m$, then
    \begin{align}\label{equ:a_r}
        \begin{split}
            & H_m(\alpha)-H_m(\alpha') \\
              = &\frac{a_{r+1}-a_r}{4}\left(\frac{1-(-1)^{a_{r+1}+a_r}}{2(a_{r+1}-a_r)}+1+(-1)^{a_{r+1}+a_r}-(a_{r+1}-a_r)\right)\\
              = &\begin{cases}
              < 0,  \ \textmd{ if } a_{r+1}-a_r>2;\\
              =0,  \ \textmd{ if } a_{r+1}-a_r=2.
             \end{cases}
        \end{split}
    \end{align}
        
   For any positive integer $r$ and any $\alpha=(a_1,\dots,a_r)\in \mathbb{Z}^r$, let $l(\alpha)$ be the number of elements in $\{i \mid a_i=i,1\leq i\leq r\}$.
   Take $\beta=(b_1,\dots,b_m)\in W_m$ for some $1\leq m\leq n$ such that $H_m(\beta)=M$ and $l(\beta)=\max\{l(\alpha)\mid \exists 1\leq r\leq n, \alpha\in W_r,H_r(\alpha)=M\}$. Assume $l(\beta)<m$. Then $m<n$.
   Note that if $b_i=i$ for some $1\leq i\leq m$ then $b_j=j$ for all $1\leq j\leq i$. Then $b_i=i$ for all $1\leq i\leq l(\beta)$.
   If $l(\beta)=0$, then $b_1\geq 2$. Let $\beta'=(1,b_1,\dots,b_m)$, then $\beta'\in W_{m+1}$ and $l(\beta')>l(\beta)$. 
   It follows from (\ref{equ:H_m(a)=H_m+1(0,a)}) and (\ref{equ:a_1}) that $H_m(\beta)=H_{m+1}(0,b_1,\dots,b_m)\leq H_{m+1}(\beta')$.
   This leads to a contradiction.
   If $l(\beta)\geq 1$ and $b_{l(\beta)+1}>l(\beta)+2$, let $\beta'=(b'_{1},\dots,b'_m)\in W_{m}$, where $b'_{l(\beta)}=b_{l(\beta)}+1$ and $b'_{j}=b_j$ for any $1\leq j\neq l(\beta) \leq m$. 
   It follows from (\ref{equ:a_1}) and (\ref{equ:a_r}) that $H_m(\beta)< H_{m}(\beta')$. This leads to a contradiction.
   If $l(\beta)\geq 1$ and $b_{l(\beta)+1}=l(\beta)+2$, let $\beta'=(1,2,\dots,l(\beta),l(\beta)+1,b_{l(\beta)+1},\dots,b_m)$, then $\beta'\in W_{m+1}$ and $l(\beta')>l(\beta)$.
    For any $1\leq i\leq l(\beta)$, let $\beta'_i=(b_{i1},\dots,b_{im})$, where 
    \[
    b_{ij}=
    \begin{cases}
        b_j+1,  \ \textmd{ if } l(\beta)-i<j\leq l(\beta); \\
        b_j,  \ \textmd{ if } j\leq l(\beta)-i \textmd{ or } j> l(\beta).
    \end{cases}
    \]
   It follows from (\ref{equ:H_m(a)=H_m+1(0,a)}), (\ref{equ:a_1}) and (\ref{equ:a_r}) that $H_m(\beta)=H_m(\beta'_1)=H_m(\beta'_2)=\dots =H_m(\beta'_{l(\beta)})=H_{m+1}(0,2,\dots,l(\beta)+1,b_{l(\beta)+1},\dots,b_m)=H_{m+1}(\beta') $. This leads to a contradiction.
   So $l(\beta)=m$ and $\beta=(1,2,\dots,m)$. Then 
   $H_m(\beta)=\frac{n^3-(n-m)^3-m}{3}$. It follows that
   \[
M=\max\{\frac{n^3-(n-m)^3-m}{3}\mid m \in \{1,2,\dots,n\}\}=\frac{n^3-n}{3}.
   \]
    
\end{proof}

\subsection{ Castelnuovo–Mumford regularity}

\begin{definition}
    Let $M$ be a finitely generated graded module over the polynomial ring $S = \mathbb{K}[x_1,...,x_n]$. The graded Betti number of $M$ is $\beta_{ij}^S(M)=\dim_{\mathbb{K}}\mathrm{Tor}_i^S(M,\mathbb{K})_j$, where the vector space $\mathrm{Tor}_i^S(M,\mathbb{K})_j$ is the degree $j$ component of the graded vector space $\mathrm{Tor}_i^S(M,\mathbb{K})$.
    The Castelnuovo–Mumford regularity of $M$ is $\mathrm{reg}_S(M)=\sup\{j-i|\beta_{ij}^S(M)\neq 0\}$.
\end{definition}

\begin{remark}
    The number $\beta_{0j}^S(M)$ is the number of elements of degree $j$ required among the minimal generators of $M$.
\end{remark}

\begin{proposition}\label{prop: complex}
    Let $M$ be a graded module over $S=\mathbb{K}[x_1,\dots,x_r]$. The graded Betti number $\beta_{ij}^S(M)$ is the dimension of the homology, at the term $M_{j-i}\otimes \wedge^i \mathbb{K}^{r}$, of the complex
\[
        0 \to M_{j-r}\otimes \wedge^{r} \mathbb{K}^{r}\to \dots \to M_{j-i}\otimes \wedge^{i} \mathbb{K}^{r} \to  \dots\to M_{j}\otimes \wedge^{0} \mathbb{K}^{r}\to 0.
\]
\end{proposition}
\begin{proof}
    See \cite[Proposition 2.7]{Eisenbud}.
\end{proof}

\begin{proposition}\label{prop: reg}
    For every homogenous ideal $I\subseteq S=\mathbb{K}[x_1,\dots,x_d]$ that is generated in degree at most $\kappa$, then $\mathrm{reg}_S(I)\leq d(\kappa-1)+1$ if $d\leq 3$, and $\mathrm{reg}_S(I)\leq [3\kappa^2(\kappa-1)]^{2^{d-4}}+1$ if $d\geq 4$.
\end{proposition} 
\begin{proof}
    See \cite[Example 3.6]{Chardin:regularity}.
\end{proof}

\subsection{Good filtrations and reductive group actions}\ 

Let $G$ be a connected and  reductive group over $\mathbb{K}$. By a $G$-module we mean a rational $G$-module. Fix a Borel subgroup $B$ and a maximal torus $T\subseteq B$. 
Let $X(T)$ be the character group of $T$. We choose an ordering of roots such that the roots in $B$ are negative. 
    For every dominant weight $\lambda\in X(T)^+$, $H^0(\lambda)$ denotes an induced module $\mathrm{Ind}_B^G \mathbb{K}_{\lambda}$, where $\mathbb{K}_{\lambda}=\mathbb{K}$ is the $B$-module on which $T$ acts via $\lambda$.

\begin{definition}
     A rational $G$-module $M$ has a good filtration if there is an ascending chain $0= M_0\subseteq M_1\subseteq\dots\subseteq M$ of submodules with $\cup_{i\in \mathbb{N}} M_i = M$ such that for any $i>0$ the module $M_i/M_{i-1}$ is either zero or isomorphic to $H^0(\lambda_i)$ for some $\lambda_i\in X(T)^+$.
\end{definition}

\begin{proposition}\label{prop: good filtration}
    Let $M,N$ be $G$-modules. Then
    \begin{enumerate}
        \item If $M$ has a good filtration, then $H^i(G,M\otimes H^0(\lambda))=0$ for all $i>0$ and $\lambda\in X(T)^+$.
        \item If $M$ has countable dimension and $H^1(G,M\otimes H^0(\lambda))=0$ for all $\lambda\in X(T)^+$, then $M$ has a good filtration,
        \item If $M$ and $N$ have good filtrations, then $M\otimes N$, considered under the diagonal action of $G$, also has a good filtration.
        \item Let $V=\mathbb{K}^n$ and $G$ be one of groups $GL(V)$, $SL(V)$, $SO(V)$ or $Sp(V)$. Then $\wedge(V)$, $\wedge(V^*)$, $S(V^*\otimes V)$  have good filtrations, where $\wedge$ and $S$ denote the exterior and symmetric powers.
    \end{enumerate}
\end{proposition}
\begin{proof}
    (1) See \cite[Proposition 1.2a]{Donkin:good fil}. 
    (2) See \cite[Proposition 1.2a]{Donkin:good fil}. 
    (3) This result was first obtained for $\mathrm{char}\mathbb{K}\gg 0$ by J. Wang [\ref{Wang}], with improvements to the prime $p$ by S. Donkin [\ref{Donkin:rep}]  (under some small restrictions), and in general by O. Mathieu \cite[Theorem 1]{Mathieu:Fil}.
    (4) It follows from \cite[4.9]{Andersen:ind rep} that if $G=GL(V)$ or $SL(V)$ then $\wedge(V)$ and $\wedge(V^*)$ have good filtrations, then $S(V^*\otimes V)$ has a good filtration by \cite[4.3(2)]{Andersen:ind rep}. So if $G=SO(V)$ or $Sp(V)$ then $\wedge(V)$ , $\wedge(V^*)$ and $S(V^*\otimes V)$ have good filtrations by \cite[3.2.6]{H-M:good fil subgroup}. 
\end{proof}

\begin{remark}
    Note that $H^0(0)=\mathbb{K}$ is a trivial $G$-module. Hence if $M$ has a good filtration, then $H^i(G,M)=0$ for all $i>0$.
\end{remark}

\begin{proposition}\label{prop: semisimple}
     If $V_{1},\dots,V_r$ are semisimple $G$-modules, and $i_{1},\dots,i_r\geq 0$ integers with
\[
\sum_{j=1}^r i_{j}(\dim(V_{j})-i_{j})<\mathrm{char}\ \mathbb{K},
\]
then $\otimes_{j=1}^r\wedge^{i_j}V_{j}$ is semisimple.
 \end{proposition}   
\begin{proof}
    See \cite[p.26]{Serre}.
\end{proof}

\subsection{Schur functor and The Decomposition of the Exterior Algebra}

\begin{definition}
    A partition is an element $\lambda=(\lambda_1,\dots,\lambda_r)\in \mathbb{N}^r$ for some $r\geq 1$, such that $\lambda_1\geq \lambda_2\geq\dots\geq\lambda_r$. The weight of the partition $\lambda$, denoted by $|\lambda|$, is the sum $\sum_{i=1}^r\lambda_i$. If  $|\lambda|=n$, $\lambda$ is said to be a partition of $n$. The number of non-zero terms of $\lambda$ is called the length of $\lambda$. The conjugate $\tilde{\lambda}$ of $\lambda$ is an element $\tilde{\lambda}=(\tilde{\lambda}_1,\dots,\tilde{\lambda}_s)\in \mathbb{N}^s$, where $\tilde{\lambda}_j$ is the number of terms of $\lambda$ which are greater than or equal to $j$.
\end{definition}

\begin{definition}
    If $F$ is a free module over a commutative ring $R$, and $\lambda=(\lambda_1,\dots,\lambda_q)$ is a partition, we use the following  notation:
    \begin{align*}
        \wedge_{\lambda}F & = \wedge^{\lambda_1}F\otimes_R\dots \otimes_R \wedge^{\lambda_q}F \\
        S_{\lambda}F & = S^{\lambda_1}F\otimes_R\dots \otimes_R S^{\lambda_q}F \\
        D_{\lambda}F & = D^{\lambda_1}F\otimes_R\dots \otimes_R D^{\lambda_q}F 
    \end{align*}
where $\wedge,S$ and $D$ denote the exterior, symmetric and divided powers. 
Let $GL(F)$ be the set of all $R$-linear automorphisms of $F$.
We can define the morphisms of $GL(F)$-modules
\[
d_{\lambda}(F):\wedge_{\lambda}F\to  S_{\Tilde{\lambda}}F, d'_{\lambda}(F):D_{\lambda}F\to  \wedge_{\Tilde{\lambda}}F 
\]
associated to the partition $\lambda$ and the free module $F$. For more details about the definitions of divided power and morphisms $d_{\lambda}(F),d'_{\lambda}(F)$, we refer to \cite{Akin:Schur functor}.
The image of $d_{\lambda}(F)$ is defined  to be the Schur functor of $F$ with respect to the partition $\lambda$, and is denoted  by $L_{\lambda}F$. The image of $d'_{\lambda}(F)$ is defined to be the  coSchur functor of $F$ with respect to the  partition $\lambda$, and is denoted by $K_{\lambda}F$. $L_{\lambda}F(K_{\lambda}F)$ is also called the Schur (coSchur) functor  of shape $\lambda$.

\end{definition}

\begin{proposition}\label{prop: decomposition}
    Let $F,G$ be finitely generated free $R$-modules. The  exterior  algebra $\wedge (F\otimes G)$ is  naturally a $GL(F)\times GL(G)$-module. Then there is a universal  filtration of $\wedge^k (F\otimes G)$ by $GL(F)\times GL(G)$-modules  whose associated graded object is $\sum_{|\lambda|=k} L_{\lambda}F\otimes K_{\lambda}G$. 
\end{proposition}
\begin{proof}
    See \cite[Theorem \uppercase\expandafter{\romannumeral3}.2.4]{Akin:Schur functor}.
\end{proof}

\section{Main theorems}\label{sec:main theorems}
Suppose $n\geq 2, \ d\geq 1$ are   positive integers. Let 
\begin{displaymath}
O_n(\mathbb{K}):=\{A\in M_n(\mathbb{K})\:|\: A A^t=I_n\}
\end{displaymath}
be the orthogonal group, and if $n$ is even, let
\begin{displaymath}
Sp_n(\mathbb{K}):=\{A\in M_n(\mathbb{K})\:|\: A^t J A =I_n\}
\end{displaymath}
be the symplectic group, where $J=\begin{pmatrix}
0& I_{\frac{n}{2}}\\ 
-I_{\frac{n}{2}}&0
\end{pmatrix}$. 
 
 Throughout this section, $G$ is  one of the following groups:
\begin{displaymath}
G=
\begin{cases}
GL_n(\mathbb{K}), \\
O_n(\mathbb{K}), \\ 
SO_n(\mathbb{K}),\\
Sp_n(\mathbb{K}), & n \textmd{   even}. 
\end{cases}
\end{displaymath}

Let $B$ be a Borel subgroup of $G$, $T\subseteq B$ be a maximal torus of $G$ and $W=N_G(T)/T$ be the Weyl group. Let $X(T)$ be the character group of $T$. We choose an ordering of roots such that the roots in $B$ are negative. 
As usual,  $\mathfrak{g}$ denotes the Lie algebra of $G$. Explicitly, 
\begin{displaymath}
\mathfrak{g}=\begin{cases}
\mathfrak{gl}_n, \ \textmd{ if } G \ \textmd{ is } GL_n(\mathbb{K});\\
\mathfrak{so}_n, \ \textmd{ if } G \ \textmd{ is } O_n(\mathbb{K}) \textmd{ or } SO_n(\mathbb{K}); \\ 
\mathfrak{sp}_n, \ \textmd{ if } G \ \textmd{ is } Sp_n(\mathbb{K}),
\end{cases}
\end{displaymath}
where we fix the realizations of the Lie algebras as matrices:
\begin{displaymath}
\mathfrak{so}_n:=\{A\in M_n(\mathbb{K}\:)|\: A+A^t=0\}, \ \mathfrak{sp}_n:=\{A\in M_n(\mathbb{K})\:|\: A^t J+ J A=0\}.
\end{displaymath}

 The Lie algebra $\mathfrak{t}$ of $T$ is  a Cartan subalgebra of $\mathfrak{g}$. Via the diagonal adjoint  representation, the group $G$ acts on $\mathfrak{g}^d$ and this induces  an action of $W$ on $\mathfrak{t}^d$. 
 
 For $1\leq k\leq d$, \  $1\leq i,j\leq n$, let $x(k)_{ij}$ be the polynomial function of $M_n(\mathbb{K})^d$ whose value at a point $(A_1,\cdots, A_d)\in M_n(\mathbb{K})^d$ is the $(i,j)$-entry of the matrix $A_k\in M_n(\mathbb{K})$. Let 
 \[
 R=\mathbb{K}[M_n(\mathbb{K})^d],\ \Bar{R}=\mathbb{K}[\mathfrak{g}^d],\ I'=\ker(R\to \Bar{R}).
 \]

 Over the  ring $\mathbb{K}[M_n(\mathbb{K})^d]$, consider the  "generic" $n\times n$ matrices $X(1)$, $X(2)$, $\cdots$, $X(d)$, such that  the $(i,j)$-entry of $X(k)$ is $x(k)_{ij}$. 
 Let $E$ be the $\mathbb{K}$-linear space spanned by all of the entries of the matrices $[X(k), X(l)]:=X(k)X(l)-X(l)X(k), \ 1\leq k<l\leq d$, and $I$ be the homogeneous ideal of $\mathbb{K}[M_n(\mathbb{K})^d]$ generated by $I'$ and $E$. We define the quotient ring 
\[
\mathbb{K}[\mathfrak{C}^d_{\mathfrak{g}}]:=\mathbb{K}[M_n(\mathbb{K})^d]/I=\mathbb{K}[\mathfrak{g}^d]/\Bar{I},
\]
 where $\Bar{I}=I/I'$. This ring can be viewed as the coordinate ring of the commuting scheme $\mathfrak{C}^d_{\mathfrak{g}}$.

Let $G$ act on $M_n(\mathbb{K})^d$ via simultaneous conjugation: $g\cdot(A_1,\cdots, A_d)=(gA_1 g^{-1},\cdots, gA_d g^{-1})$, \ for $g\in G$, \ $(A_1,\cdots, A_d)\in M_n(\mathbb{K})^d$. Then the inclusion $\mathfrak{g}^d\subset M_n(\mathbb{K})^d$ induces a restriction homomorphism 
\[
\varphi: \mathbb{K}[M_n(\mathbb{K})^d]^G\rightarrow \mathbb{K}[\mathfrak{C}^d_{\mathfrak{g}}]^G.
\]

Since obviously $I$ is $G$-invariant, we have the induced action of $G$ on $\mathbb{K}[\mathfrak{C}^d_{\mathfrak{g}}]$. Moreover, the action preserves the degrees on $\mathbb{K}[\mathfrak{C}^d_{\mathfrak{g}}]$, so that the invariant subring $\mathbb{K}[\mathfrak{C}^d_{\mathfrak{g}}]^G$ is still a graded $\mathbb{K}$-algebra whose degree zero part is equal to $\mathbb{K}$. Any polynomial function on $\mathfrak{g}^d$ restricts to a polynomial function on $\mathfrak{t}^d$ through the inclusion $\mathfrak{t}^d\subset \mathfrak{g}^d$, and the restriction homomorphism $\mathbb{K}[\mathfrak{g}^d]\rightarrow \mathbb{K}[\mathfrak{t}^d]$ factors through $\mathbb{K}[\mathfrak{C}^d_{\mathfrak{g}}]$. This induces  the following  restriction homomorphism between the invariant rings:
\begin{displaymath}
\Phi: \mathbb{K}[\mathfrak{C}^d_{\mathfrak{g}}]^G \rightarrow \mathbb{K}[\mathfrak{t}^d]^W.
\end{displaymath}

We will need the following lemma. 

\begin{lemma}\label{lemma:submod good fil}
    Given positive integers $m,m',\alpha,\beta$. Let $G$ be connected and act on  $V=\mathbb{K}^n$ in  the  natural  way.
    If $\mathrm{char}\ \mathbb{K}>2\alpha (n-1)+\frac{m(n^3-n)}{3}$, then any $G$-submodule of $((V^*\otimes V)^{\otimes \alpha}\otimes\wedge^{\beta}((V^*\otimes V)^{\oplus m}))^{\oplus m'}$ has a good filtration.
\end{lemma}
\begin{proof}
    Since the $G$-module $((V^*\otimes V)^{\otimes \alpha}\otimes\wedge^{\beta}((V^*\otimes V)^{\oplus m}))^{\oplus m'}$ is isomorphic to $\oplus_{k=1}^{m'}\oplus_{i_1+\dots+i_m=\beta} ((V^*\otimes V)^{\otimes \alpha})\otimes (\otimes_{j=1}^m  \wedge^{i_j}(V^*\otimes V))$, any $G$-submodule $M$ of  $((V^*\otimes V)^{\otimes \alpha}\otimes\wedge^{\beta}((V^*\otimes V)^{\oplus m}))^{\oplus m'}$ has a filtration $0=M_0\subseteq M_1\subseteq \dots \subseteq M_t=M$ such that $M_i/M_{i-1}$ is a $G$-submodule of $((V^*\otimes V)^{\otimes \alpha})\otimes (\otimes_{j=1}^m  \wedge^{i_j}(V^*\otimes V))$ for some $i_1+\dots+i_m=\beta$.
    
    By Proposition \ref{prop: good filtration}, it suffices to show that given positive integers $i_{1},\dots,i_m$, for any $0\leq r\leq m$, any $G$-submodule of 
    \[
    N^r_{\lambda_1,\dots,\lambda_r}:=
    ((V^*\otimes V)^{\otimes\alpha})\otimes (\otimes_{i=1}^r (L_{\lambda_i}V^*\otimes K_{\lambda_i}V))\otimes (\otimes_{j=r+1}^{m}  \wedge^{i_j}(V^*\otimes V))
    \]
    has a good filtration for any partitions $\lambda_1,\dots,\lambda_r$. If $r=m$, for any partitions $\lambda_{1},\dots,\lambda_{m}$, we may assume 
    \[
    \lambda_i=(\underbrace{a_{i1},\dots,a_{i1}}_{b_{i1}},\underbrace{a_{i2},\dots,a_{i2}}_{b_{i2}},\dots,\underbrace{a_{il_i},\dots,a_{il_i}}_{b_{il_i}}), \ 1\leq i\leq m,
    \]
    where $a_{i1}>\dots >a_{il_i}>0$ for any $1\leq i\leq m$. Then for any $1\leq i\leq m$,
    \[
    \tilde{\lambda}_i=(\underbrace{\sum_{j=1}^{l_i}b_{ij},\dots,\sum_{j=1}^{l_i}b_{ij}}_{a_{il_i}},\underbrace{\sum_{j=1}^{l_i-1}b_{ij},\dots,\sum_{j=1}^{l_i-1}b_{ij}}_{a_{i,l_i-1}-a_{il_i}},\dots,\underbrace{b_{i1},\dots,b_{i1}}_{a_{i1}-a_{i2}}).
    \]
    Since for any $1\leq i\leq m$ the $G$-module $L_{\lambda_i}V^*$ is a quotient of $\wedge_{\lambda_i}V^*$ and $K_{\lambda_i}V\subseteq \wedge_{\tilde{\lambda}_i}V$, we may assume $a_{i1},\sum_{j=1}^{l_i}b_{ij}\leq n$ for any $1\leq i\leq m$ and $N^m_{\lambda_1,\dots,\lambda_m}$ is a quotient of some submodule of $\otimes_{i=1}^m(\wedge_{\lambda_i}V^*\otimes \wedge_{\tilde{\lambda}_i}V)$.
    By Lemma \ref{lemma:max}, we have
    \begin{align*}
        & \mathrm{char}\ \mathbb{K}  > 2\alpha(n-1)+\frac{m(n^3-n)}{3} \geq 2\alpha(n-1)   \\
         & +\sum_{i=1}^m\left(\sum_{k=1}^{l_i} b_{ik}a_{ik}(n-a_{ik})+(a_{ik}-a_{i,k+1})\sum_{j=1}^{k}b_{ij}(n- \sum_{j=1}^{k}b_{ij})\right)\\
    \end{align*}
    where $a_{1,l_1+1}=\dots=a_{m,l_m+1}=0$.
    Note that $V$ and $V^*$ are simple.
    By Proposition \ref{prop: good filtration} and Proposition \ref{prop: semisimple}, we see that $N^m_{\lambda_1,\dots,\lambda_m}$ is semisimple and has a good filtration. Hence any $G$-submodule of $N^m_{\lambda_1,\dots,\lambda_m}$ has a good filtration.

    Assume it is true for every integer $k$ with $0\leq r<k\leq m$. For any partitions $\lambda_1,\dots,\lambda_r$, 
    by Proposition \ref{prop: decomposition}, the $G$-module $N^r_{\lambda_1,\dots,\lambda_r}$ has a filtration $0=N_0\subseteq N_1\subseteq \dots \subseteq N_s=N^r_{\lambda_1,\dots,\lambda_r}$ such that $N_{j}/N_{j-1}$ is $N^{r+1}_{\lambda_1,\dots,\lambda_r,\eta_j}$ for some partition $\eta_j$. Applying the induction hypothesis to $N^{r+1}_{\lambda_1,\dots,\lambda_r,\eta_j},1\leq j\leq s$, we obtain that any $G$-submodule of $N_j$ has a good filtration for any $1\leq j\leq s$ by Proposition \ref{prop: good filtration}. This completes the proof.

\end{proof}

\begin{theorem}\label{main thm: good fil}
If $\mathrm{char}\ \mathbb{K}>2(\mathrm{reg}_{\Bar{R}}(\Bar{I})-1)(n-1)+\frac{d(n^3-n)}{3}$ and $G$ is connected, then $\Bar{I}$ and $I$ have good filtrations.
\end{theorem}
\begin{proof}
    First, we will prove that $\Bar{I}$  has a good filtration.
    Let $\Bar{R}=\oplus_{i=0}^{\infty}\Bar{R}_i$, $\Bar{I}=\oplus_{i=2}^{\infty}\Bar{I}_i$, and $\alpha=\mathrm{reg}_{\Bar{S}}(\Bar{I})$. Let $G$ act on  $V=\mathbb{K}^n$ in  the  natural  way.
    Note that $\alpha\geq 2$. If $2\leq i\leq \alpha$, since the $G$-module $E$ is a quotient of $\oplus_{1\leq k<l\leq d}(V^*\otimes V)$, we see that the morphism of $G$-modules
    \[
    \oplus_{j\in \Lambda_i}((V^*\otimes V)^{\otimes (i-1)})\to \Bar{S}_{i-2}\otimes E\to \Bar{I}_i
    \]
    is surjective for some finite set $\Lambda_i$. Note that $\mathrm{char}\ \mathbb{K}>2(i-1)(n-1)$. Then $\Bar{I}_i$ has a good filtration by Proposition \ref{prop: good filtration} and Proposition \ref{prop: semisimple}.
    
    Assume that $\Bar{I}_k$ has a good filtration for every integer $k$ with $\alpha\leq k<m$. For any $i\geq 1$, since the $G$-module $ \wedge^i\Bar{R}_1$ is a quotient of $\wedge^i((V^*\otimes V)^{\oplus d})$, $\wedge^i\Bar{R}_1$ has a good filtration by Proposition \ref{prop: good filtration} and Lemma \ref{lemma:submod good fil}.
    If $m>\alpha+\dim(\Bar{R})$, by Proposition \ref{prop: complex}, the complex
    \begin{equation}\label{complex:m>}
        0\to \Bar{I}_{m-\dim(\Bar{R})}\otimes\wedge^{\dim(\Bar{R})}\Bar{R}_1\to\cdots\to \Bar{I}_m\to 0
    \end{equation}
    is exact, where $\Bar{R}_1\cong \mathbb{K}^{\dim(\Bar{R})}$. It is easy to verify that (\ref{complex:m>}) is a complex of $G$-modules. 
    Applying the induction hypothesis, we obtain that $\Bar{I}_m$ has a good filtration by (\ref{complex:m>}) and Proposition \ref{prop: good filtration}.
    If $\alpha< m\leq \alpha+\dim(\Bar{R})$, by Proposition \ref{prop: complex}, there is an exact sequence of $G$-modules
    \begin{equation}\label{complex:<m<}
        0\to M\to \Bar{I}_{\alpha}\otimes\wedge^{m-\alpha}\Bar{R}_1\to\cdots\to \Bar{I}_m\to 0,
    \end{equation}
    where $M=\ker(\Bar{I}_{\alpha}\otimes\wedge^{m-\alpha}\Bar{R}_1\to \Bar{I}_{\alpha+1}\otimes\wedge^{m-\alpha-1}\Bar{R}_1)$. 
    Since the morphism of $G$-modules 
    \[
    \oplus_{j\in \Lambda_\alpha}((V^*\otimes V)^{\otimes (\alpha-1)})\otimes\wedge^{m-\alpha}((V^*\otimes V)^{\oplus d})\to\Bar{I}_{\alpha}\otimes\wedge^{m-\alpha}\Bar{R}_1
    \]
    is surjective, we see that $M$ has a good filtration by Lemma \ref{lemma:submod good fil}. 
    Applying the induction hypothesis, we obtain that $\Bar{I}_m$ has a good filtration by (\ref{complex:<m<}) and Proposition \ref{prop: good filtration}.

    Finally, we will prove that $I$  has a good filtration. Note that the $G$-module $\Bar{I}$ is the quotient of $I$ by $ I'$. 
    It suffices to show that $I'$  has a good filtration by Proposition \ref{prop: good filtration}. Let $I'=\oplus_{i=1}^{\infty}I'_i$.
    Note that $I'$ is the ideal of $R$ generated by $I'_1$ and the Koszul complex
    \begin{equation}\label{conplex:I'}
        0\to R\otimes\wedge^{\dim(I'_1)}I'_1\to R\otimes\wedge^{\dim(I'_1)-1}I'_1\to\cdots\to R\otimes I'_1\to I'\to 0
    \end{equation}
    is exact. It is easy to verify that (\ref{conplex:I'}) is a complex of $G$-modules. 
    Since $\mathrm{char}\ \mathbb{K}>2(n-1)$, $(V^*\otimes V)^{\oplus d}$ is semisimple by Proposition \ref{prop: semisimple}.
    For any $i\geq 1$, since $I'_1$ is $G$-submodule of $(V^*\otimes V)^{\oplus d}$, $\wedge^i I'_1$ is $G$-submodule of $\wedge^i((V^*\otimes V)^{\oplus d})$ and has a good filtration by Lemma \ref{lemma:submod good fil}. 
    Note that $R\cong (S(V^*\otimes V))^{\otimes d}$ has a good filtration by Proposition \ref{prop: good filtration}.
    So $I'$  has a good filtration by (\ref{conplex:I'}) and Proposition \ref{prop: good filtration}.
\end{proof}

Now consider the composition of homomorphisms:
\begin{displaymath}
\mathbb{K}[M_n(\mathbb{K})^d]^{G }\xrightarrow{ \ \varphi \ } \mathbb{K}[\mathfrak{C}^d_{\mathfrak{g}}]^{G} \xrightarrow{\ \Phi \ } \mathbb{K}[\mathfrak{t}^d]^W.
\end{displaymath}

\begin{theorem}\label{main thm}
If $\mathrm{char}\ \mathbb{K}>(12)^{2^{d\cdot\dim(\mathfrak{g})-4}}(2n-2)+\frac{d(n^3-n)}{3}$, then the restriction homomorphism $\Phi: \mathbb{K}[\mathfrak{C}^d_{\mathfrak{g}}]^G\rightarrow \mathbb{K}[\mathfrak{t}^d]^W$ is an isomorphism of $\mathbb{K}$-algebras. 
\end{theorem}
\begin{proof}
If $G=GL_n(\mathbb{K})$, let $C_P$ be the subalgebra of $\mathbb{K}[\mathfrak{C}^d_{\mathfrak{g}}]^G$ generated by the coefficients of the characteristic polynomial of polynomials in generic commuting matrices. It follows from \cite[Theorem 1]{Vaccarino} that $\Phi$ induces an isomorphism $C_P\to\mathbb{K}[\mathfrak{t}^d]^{W}$. By Theorem \ref{main thm: good fil}, we have $H^1(G,I)=0$, then $\mathrm{Im}(\varphi)=\mathbb{K}[\mathfrak{C}^d_{\mathfrak{g}}]^G$, then $C_P=\mathbb{K}[\mathfrak{C}^d_{\mathfrak{g}}]^G$ by \cite{Donkin:matrix invariants}, so $\Phi:\mathbb{K}[\mathfrak{C}^d_{\mathfrak{g}}]^G\to \mathbb{K}[\mathfrak{t}^d]^W$ is an isomorphism of $\mathbb{K}$-algebras.

If $G$ is one of the following groups:
\begin{displaymath}
G=
\begin{cases}
O_n(\mathbb{K}), \\ 
SO_n(\mathbb{K}), & n \textmd{   odd},\\
Sp_n(\mathbb{K}), & n \textmd{   even}, 
\end{cases}
\end{displaymath}
 it follows from \cite[Theorem 3.4]{S-X-X} and its proof that $\Phi$ induces a $\mathbb{K}$-linear isomorphism $\mathrm{Im}(\varphi)\to\mathbb{K}[\mathfrak{t}^d]^{W}$. By Theorem \ref{main thm: good fil}, we have $H^1(G,I)=0$, then $\mathrm{Im}(\varphi)=\mathbb{K}[\mathfrak{C}^d_{\mathfrak{g}}]^G$, so $\Phi:\mathbb{K}[\mathfrak{C}^d_{\mathfrak{g}}]^G\to \mathbb{K}[\mathfrak{t}^d]^W$ is an isomorphism of $\mathbb{K}$-algebras. 
 
 If $G'=SO_n(\mathbb{K})$ and $n$ is even, let  $G=O_n(\mathbb{K})$ and $W^{\prime}=N_{G^{\prime}}(T)/T$. Note both $G^{\prime}\subset G$ and $W^{\prime}\subset W=N_{G}(T)/T$ are subgroups of index two. We have the following commutative diagram.
\begin{displaymath}
\begin{diagram}
\mathbb{K}[\mathfrak{C}^d_{\mathfrak{g}}]^{G^{\prime}} &\rTo^{\ \ \ \ \Phi \ \ \ \ } &\mathbb{K}[\mathfrak{t}^d]^{W^{\prime}}\\
\uInto_{} & &\uInto_{}\\
\mathbb{K}[\mathfrak{C}^d_{\mathfrak{g}}]^G & \rTo^{\ \ \ \ \Phi \ \ \ \ }&\mathbb{K}[\mathfrak{t}^d]^W
\end{diagram}
\end{displaymath}

Take $w_0\in W$ which generates the quotient group $W/W^{\prime}\simeq \mathbb{Z}/2\mathbb{Z}$. Note $W/W^{\prime}$ acts naturally on $\mathbb{K}[\mathfrak{t}^d]^{W^{\prime}}$ and we have the eigen-subspace decomposition
\[\mathbb{K}[\mathfrak{t}^d]^{W^{\prime}}=\mathbb{K}[\mathfrak{t}^d]^{W^{\prime}}_{(0)}\oplus \mathbb{K}[\mathfrak{t}^d]^{W^{\prime}}_{(1)},\]
where  $\mathbb{K}[\mathfrak{t}^d]^{W^{\prime}}_{(0)}=\mathbb{K}[\mathfrak{t}^d]^W$ is the invariant part and $\mathbb{K}[\mathfrak{t}^d]^{W^{\prime}}_{(1)}=\{v\in \mathbb{K}[\mathfrak{t}^d]^{W^{\prime}}\mid w_0 v=-v\}$. Similarly, the action of $G/G^{\prime}\simeq \mathbb{Z}/2\mathbb{Z} $ on $\mathbb{K}[\mathfrak{C}^d_{\mathfrak{g}}]^{G^{\prime}}$ induces the eigen-subspace decomposition $\mathbb{K}[\mathfrak{C}^d_{\mathfrak{g}}]^{G^{\prime}}=\mathbb{K}[\mathfrak{C}^d_{\mathfrak{g}}]^{G^{\prime}}_{(0)}\oplus \mathbb{K}[\mathfrak{C}^d_{\mathfrak{g}}]^{G^{\prime}}_{(1)}$, with  $\mathbb{K}[\mathfrak{C}^d_{\mathfrak{g}}]^{G^{\prime}}_{(0)}=\mathbb{K}[\mathfrak{C}^d_{\mathfrak{g}}]^G$  and $\mathbb{K}[\mathfrak{C}^d_{\mathfrak{g}}]^{G^{\prime}}_{(1)}=\{v\in \mathbb{K}[\mathfrak{C}^d_{\mathfrak{g}}]^{G^{\prime}}\mid g_0 v=-v\}$, where $g_0$ is a generator of $G/G ^{\prime}$. Clearly the restriction homomorphism $\Phi$ preserves these decompositions: $\Phi(\mathbb{K}[\mathfrak{C}^d_{\mathfrak{g}}]^{G^{\prime}}_{(i)})\subset \mathbb{K}[\mathfrak{t}^d]^{W^{\prime}}_{(i)}$, \ $i=0,1$. 
Now consider the composition of homomorphisms:
\begin{displaymath}
\mathbb{K}[M_n(\mathbb{K})^d]^{G ^{\prime}}\xrightarrow{ \ \varphi \ } \mathbb{K}[\mathfrak{C}^d_{\mathfrak{g}}]^{G^{\prime}} \xrightarrow{\ \pi_1 \ } \mathbb{K}[\mathfrak{C}^d_{\mathfrak{g}}]^{G^{\prime}}_{(1)},
\end{displaymath}
where $\pi_1$ is the projection under the decomposition $\mathbb{K}[\mathfrak{C}^d_{\mathfrak{g}}]^{G^{\prime}}=\mathbb{K}[\mathfrak{C}^d_{\mathfrak{g}}]^{G^{\prime}}_{(0)}\oplus \mathbb{K}[\mathfrak{C}^d_{\mathfrak{g}}]^{G^{\prime}}_{(1)} $.
It follows from \cite[Theorem 4.3]{S-X-X} and its proof that $\Phi$ induces a $\mathbb{K}$-linear isomorphism $\mathrm{Im}(\pi_1\circ \varphi)\to\mathbb{K}[\mathfrak{t}^d]^{W'}_{(1)}$. By Theorem \ref{main thm: good fil}, we have $H^1(G,I)=0$, then $\mathrm{Im}(\pi_1\circ \varphi)=\mathbb{K}[\mathfrak{C}^d_{\mathfrak{g}}]^{G'}_{(1)}$. Combining  with the isomorphism $\mathbb{K}[\mathfrak{C}^d_{\mathfrak{g}}]^{G^{\prime}}_{(0)}\to\mathbb{K}[\mathfrak{t}^d]^{W^{\prime}}_{(0)}$, we finish the proof.

\end{proof}

\end{document}